\documentclass[12pt,a4paper,reqno]{amsart}
\usepackage[T1]{fontenc}
\usepackage{microtype, a4wide, textcomp,fbb}
\usepackage{
    amsmath,  amssymb,  amsthm,   amscd,
    gensymb,  graphicx, etoolbox, 
    booktabs, stackrel, mathtools    
}
\usepackage{enumitem}
\usepackage[usenames,dvipsnames]{xcolor}
\definecolor{darkblue}{rgb}{0,0,0.7}
\definecolor{darkred}{rgb}{0.7,0,0}
\usepackage[colorlinks=true, linkcolor=darkred, citecolor=darkblue, urlcolor=blue, pagebackref=true, breaklinks=true]{hyperref}
\usepackage[capitalise]{cleveref}
\usepackage{placeins}
\usepackage{relsize}
\usepackage{todonotes}
\usepackage{soul}
\newtheorem{theorem}{Theorem}[section]

\newtheorem{conjecture}[theorem]{Conjecture}
\newtheorem{corollary}[theorem] {Corollary}
\newtheorem{lemma}[theorem]{Lemma}
\newtheorem{result}[theorem]{Result}

\newtheorem{proposition}[theorem]{Proposition}
\theoremstyle{definition}
\newtheorem{definition}[theorem]{Definition}
\newtheorem{example}[theorem]{Example}

\newtheorem*{problem*}{Problem}

\usepackage{tikz}
\usetikzlibrary{arrows}
\usetikzlibrary{shapes}
\usepackage{float}
\usepackage{tabularx}
\usepackage{kantlipsum}
\usepackage{array}
\usepackage[margin=2.6cm]{geometry}
\tikzset{edgee/.style = {->,> = latex'}}
\newcolumntype{P}[1]{>{\centering\arraybackslash}p{#1}}
\newcolumntype{M}[1]{>{\centering\arraybackslash}m{#1}}

\newcommand{\one}{\textup{\texttt{1}}}
\newcommand{\zero}{\textup{\texttt{0}}}

\begin{document}

\title{Subsequence frequency in binary words}

\author{Krishna Menon}
\address{Department of Mathematics, Chennai Mathematical Institute, India}
\email{krishnamenon@cmi.ac.in}
\author{Anurag Singh}
\address{Department of Mathematics, Indian Institute of Technology (IIT) Bhilai, India}
\email{anurags@iitbhilai.ac.in}

\maketitle
\begin{abstract}
The numbers we study in this paper are of the form $B_{n, p}(k)$, which is the number of binary words of length $n$ that contain the word $p$ (as a subsequence) exactly $k$ times. 
Our motivation comes from the analogous study of pattern containment in permutations. 
In our first set of results, we obtain explicit expressions for $B_{n, p}(k)$ for small values of $k$. 
We then focus on words $p$ with at most $3$ runs and study the maximum number of occurrences of $p$ a word of length $n$ can have. 
We also study the internal zeros in the sequence $(B_{n, p}(k))_{k \geq 0}$ for fixed $n$ and discuss the unimodality and log-concavity of such sequences.
\end{abstract}

\section{Introduction}

Binary words and their subsequences is a classical topic and has been well-studied. 
For example, one direction is the study of subsequences whose indices are consecutive. 
Such subsequences are usually called \textit{factors} and there has been extensive study of words avoiding a given set of factors (see, for instance, \cite{guibas1981string,maxfac}). 
Another direction is the study of distinct subsequences that occur in a binary word. 
In \cite{snlc}, these numbers, graded by subsequence length, are shown to be log-concave. 
Maximizing the total number of distinct subsequences has been studied as well \cite{maxdistsub}. 
There have also been many other topics relating to binary words and their subsequences that have been studied (see, for instance, \cite{reconst,algosubseqcomb, longestcommon}).

In this paper, we focus on counting binary word according to the number of its subsequences that match a given word. 
Our motivation comes from the analogous study of pattern containment in permutations. 
Although most of the work in permutation patterns  is about avoiding a given set of patterns, there have also been studies about permutations that contain a given pattern exactly $k$ times for $k \geq 0$. 
However, the questions for $k \geq 1$ seem to be more difficult than the $k = 0$ case, even for patterns of length $3$ (see \cite{bona1997number, bona1998permutations,freqseq,fulmek2003enumeration,mansour2002counting, noonan1996number,noonan1996enumeration}). 
Subsequences in binary words have also been used in the study of permutation patterns, for example, see \cite{grass,invseq,grassid,4k}. 
We note that similar questions have also been tackled for pattern containment in words over a finite (totally ordered) alphabet \cite{burstein2002counting,mansourbook}. 

The numbers we study in this paper are of the form $B_{n, p}(k)$, which is the number of binary words of length $n$ that contain the word $p$ exactly $k$ times. 
In \Cref{prelim}, we set up notations and make relevant definitions. 
We then obtain explicit expressions for $B_{n, p}(k)$ for small values of $k$ in \Cref{expsec}. 
In \Cref{maxoccandintzerosec}, for certain binary words $p$, we study the maximum number of occurrences of $p$ a binary word of length $n$ can have. 
We also study internal zeros in the sequence $(B_{n, p}(k))_{k \geq 0}$ for fixed $n$. 
Furthermore, we discuss the unimodality and log-concavity of such sequences. 
We end this article with directions for future research in \Cref{concsec}.

\section{Preliminaries}\label{prelim}

A \emph{binary word} is a finite sequence $w = w_1 w_2 \cdots w_n$ where $w_i \in \{\zero, \one\}$ for all $i \in [n]$. 
Here $n$ is called the \emph{length} of $w$. 
An \emph{occurrence} of a binary word $p = p_1 p_2 \cdots p_l$ in a binary word $w = w_1 w_2 \cdots w_n$ is a subsequence of $w$ that matches $p$, i.e., a choice of indices $1 \leq i_1 < i_2 < \cdots < i_l \leq n$ such that $w_{i_1} w_{i_2} \cdots w_{i_l} = p$.

Analogous to the notation of \cite{freqseq}, we denote the number of occurrences of $p$ in $w$ by $c_p(w)$. 
For example, $c_{\one\zero}(\one\zero\zero\one\zero) = 4$. 
For any binary word $p$ and $n, k \geq 0$, we define $B_{n, p}(k)$ to be the number of binary words of length $n$ that have exactly $k$ occurrences of $p$. 
That is, $$B_{n, p}(k) = \#\{w \in \{\texttt{0}, \texttt{1}\}^n \mid c_p(w) = k\}.$$

The following results are simple consequences of the definitions.

\begin{result}
Let $p$ be a binary word of length $l$. 
For any $n \geq 0$, we have
\begin{itemize}
    \item $\sum\limits_{k = 0}^n B_{n, p}(k) = 2^n$.
    
    \item $\sum\limits_{k = 0}^n k \cdot B_{n, p}(k) = 2^{n - l} \binom{n}{l}$.
\end{itemize}
\end{result}

From the second statement above, we also get that the average number of occurrences of a binary word of length $l$ among binary words of length $n$ is $2^{-l} \binom{n}{l}$.

We denote the reverse of a binary word $p$ by $p^r$. 
The \emph{complement} of a binary word is the one obtained by changing each $\texttt{1}$ in the word to $\texttt{0}$ and vice versa. 
We denote the complement of $p$ by $p^c$. 
We say that the words $p$, $p^r$, $p^c$, and $p^{rc}$ are \emph{trivially equivalent}. 
Here $p^{rc}$ is the complement of $p^r$. 
One can check that if $p$ and $q$ are trivially equivalent, we have $B_{n, p}(k) = B_{n, q}(k)$ for all $n, k \geq 0$.

For any binary word $w$, we use $w_i$ to denote the $i^{th}$ letter in $w$. 
A \textit{run} in a binary word $w$ of length $n$ is a subsequence $w_{i+1} w_{i + 2} \cdots w_{i + k}$ of consecutive letters such that
\begin{enumerate}
    \item all letters are equal,
    \item either $i=0$ or $w_{i} \neq w_{i + 1}$, and
    \item either $i+k=n$ or $w_{i+k+1} \neq w_{i+k}$.
\end{enumerate}
Here $k$ is called the \emph{size} of the run. 
Hence, a run is a maximal subsequence of consecutive letters that are equal. 
For instance, the word $\texttt{11100001} = \one^3\zero^4\one^1$ has three runs, which are of sizes $3$, $4$, and $1$ respectively.

For use in \Cref{maxoccandintzerosec}, we recall the following definitions. 
A sequence of non-negative integers $(a_k)_{k = 0}^m$ is said to have an \emph{internal zero} if there exist $0 \leq k_1 < k_2 < k_3 \leq m$ such that $a_{k_1}, a_{k_3} \neq 0$ but $a_{k_2} = 0$. 
The sequence is said to be \emph{unimodal} if there exists an $i \in [0, m]$ such that $a_0 \leq a_1 \leq \cdots \leq a_i \geq a_{i + 1} \geq \cdots \geq a_m$. 
The sequence is said to be \emph{log-concave} if $a_i^2 \geq a_{i - 1}a_{i + 1}$ for all $i \in [m - 1]$.

\section{Binary words with few occurrences}\label{expsec}

In this section, we obtain expressions for $B_{n, p}(k)$ for $k \in [0, 4]$.

\subsection{The cases $k = 0, 1, 2$}

Let $p$ be a binary word of length $l$ that has $r$ runs, $r_i$ of which are of size $i$ for each $i \geq 1$.

\begin{proposition}\label{0count}
    For any $n \geq 0$, we have
    \begin{equation*}
        B_{n, p}(0) = \sum_{j = 0}^{l - 1} \binom{n}{j}.
    \end{equation*}
\end{proposition}

\begin{proof}
The proof is in the same lines as that of \cite[Proposition 3.22]{4k}. 
We have to count the binary words $w$ of length $n$ that do not contain the subsequence $p = p_1 p_2 \cdots p_l$. 
Suppose $j \in [0, l - 1]$ is the largest number such that $w$ contains $p_1 p_2 \cdots p_j$. 
Then, using the rightmost occurrence of $p_1p_2 \cdots p_j$, we can see that $w$ must be of the form
\begin{equation*}
    (1 - p_1)^{i_1}\ p_1\ (1 - p_2)^{i_2}\ p_2\ \cdots\ (1 - p_j)^{i_j}\ p_j\ (1 - p_{j + 1})^{i_{j + 1}}
\end{equation*}
where $i_k \geq 0$ for each $k \in [j + 1]$. 
There are $\binom{n}{j}$ such words of length $n$. 
Summing over $j \in [0, l - 1]$ gives us the required result.
\end{proof}

\begin{proposition}\label{1count}
For any $n \geq 0$, we have
\begin{equation*}
    B_{n, p}(1) = \binom{n - r + 1}{l - r + 1}.
\end{equation*}
\end{proposition}

\begin{proof}
We have to count the binary words of length $n$ that contain the subsequence $p$ exactly once. 
We do this by seeing how we can add letters to $p$ to obtain such a word. 
Consider the spaces between the letters of $p$ as \emph{slots}, including the spaces before the first and after the last letter.

If two adjacent letters of $p$ are different, then adding any letter to the slot between them will result in a word having more than one occurrence of $p$. 
If two adjacent letters of $p$ are equal, then adding a letter of the same type in the slot between them results in a word having more than one occurrence of $p$. 
However, adding letters of the opposite type in this slot will not increase the number of occurrences of $p$.

Also, we cannot add letters equal to the first letter of $p$ in the first slot nor can we add letters equal to the last letter of $p$ in the last slot. 
But we can always add letters of the opposite type before the first or after the last letter respectively.

Words constructed in such a way will have only one occurrence of $p$ since the lexicographically smallest and largest occurrence of $p$ coincide.

\begin{example}\label{1example}
Suppose $p = \one\zero\zero\zero\one\one$. 
The procedure described above is represented in the following, where the letters under the slots represent what types of letters can be inserted in them.
\begin{equation*}
    \underbracket{\quad}_{\zero}\ \one\ \underbracket{\quad}\ \zero\ \underbracket{\quad}_{\one}\ \zero\ \underbracket{\quad}_{\one}\ \zero\ \underbracket{\quad}\ \one\ \underbracket{\quad}_{\zero}\ \one\ \underbracket{\quad}_{\zero}
\end{equation*}
For example, one way to insert letters into slots is as follows.
\begin{equation*}
    \underbracket{\ \zero\ \zero\ }_{\zero}\ \one\ \underbracket{\quad}\ \zero\ \underbracket{\quad}_{\one}\ \zero\ \underbracket{\ \one\ }_{\one}\ \zero\ \underbracket{\quad}\ \one\ \underbracket{\ \zero\ \zero\ }_{\zero}\ \one\ \underbracket{\quad}_{\zero}
\end{equation*}
This gives the word $\zero\zero\one\zero\zero\one\zero\one\zero\zero\one$ that contains exactly one occurrence of $p$.
\end{example}

The observations above show that we can add letters into $(l + 1) - (r - 1)$ slots and in each slot there is only one type of letter we can add. 
Hence, the required count is the number of ways to putting $n - l$ indistinguishable balls into $l - r + 2$ distinguishable boxes, which gives us the required result (see \cite[Section 1.9]{ec1}).
\end{proof}

\begin{proposition}\label{2count}
For any $n \geq 0$, we have
\begin{equation*}
    B_{n, p}(2) =
    \begin{cases}
    r_1 \dbinom{n - r}{l - r + 1} &\text{if $l \geq 2$}\\[0.5cm]
    \dbinom{n}{2} &\text{if $l = 1$.}
    \end{cases}
\end{equation*}
\end{proposition}

\begin{proof}
The case $l = 1$ is straightforward, so we assume $l \geq 2$. 
We first describe a way to construct words of length $n$ that contain $p$ twice. 
We then show that all such words are constructed by this procedure.

Suppose $p_k$ forms a run of size $1$ in $p$. 
Start with a word $w'$ of length $n - 1$ such that $c_p(w') = 1$. 
Suppose the occurrence of $p$ in $w'$ is at indices $i_1 < i_2 < \cdots < i_l$. 
The word $w$ of length $n$ obtained by adding the letter $p_k$ immediately after $w'_{i_k}$ satisfies $c_p(w) = 2$.

\begin{example}\label{2countex}
Let $p = \one\one\zero\one\zero\zero$. 
Using the procedure described in the proof of \Cref{1count}, we express an instance of the above method using a `slot diagram'. 
We choose $p_3 = \zero$ as the run of size $1$.

$$\one\ \one\ \textcolor{red}{\zero}\ \one\ \zero\ \zero$$
$$\downarrow$$
$$\one\ \one\ \textcolor{red}{\zero\ \zero}\ \one\ \zero\ \zero$$
$$\downarrow$$
\begin{equation*}
    \underbracket{\quad}_{\zero}\ \one\ \underbracket{\quad}_{\zero}\ \one\ \underbracket{\quad}\ \textcolor{red}{\zero\ \zero}\ \underbracket{\quad}\ \one\ \underbracket{\quad}\ \zero\ \underbracket{\quad}_{\one}\ \zero\ \underbracket{\quad}_{\one}
\end{equation*}

Inserting letters into the appropriate slots, we can get the word $w = \one\zero\one\zero\zero\one\zero\one\one\zero\one$. 
The fact that $c_p(w) = 2$ is illustrated below, where the two rows of dashes indicate the occurrences.

\setcounter{MaxMatrixCols}{20}
\begin{equation*}
    \begin{matrix}
        \one & \zero & \one & \zero & \zero & \one & \zero & \one & \one & \zero & \one\\
        - &  & - & - &  & - & - &  &  & - & \\
        - &  & - &  & - & - & - &  &  & - & \\
    \end{matrix}
\end{equation*}
\end{example}

Conversely, suppose $w$ is a binary word of length $n$ that has exactly $2$ occurrences of $p$ at indices $i_1 < i_2 < \cdots < i_l$ and $j_1 < j_2 < \cdots < j_l$. 
Let $k \in [l]$ be the smallest such that $i_k \neq j_k$. 
Without loss of generality, we assume $i_k < j_k$.

We must have $i_m = j_m$ for all $m > k$. 
Otherwise, $i_1 < i_2 < \cdots < i_k < j_{k + 1} < \cdots < j_l$ gives a third occurrence of $p$. 
Also, $p_k$ must form a run of size $1$ in $p$. 
If $p_{k - 1} = p_k$, then $i_1 < i_2 < \cdots < i_{k - 2} < i_k < j_k < j_{k + 1} < \cdots < j_l$ gives a third occurrence of $p$. 
We get a similar contradiction if $p_k = p_{k + 1}$. 
Using similar ideas, we can also prove that $j_k = i_k + 1$.

This shows that deleting $w_{j_k}$ (or $w_{i_k}$) results in a word of length $n - 1$ that has exactly one occurrence of $p$. 
Hence, $w$ is obtained by the procedure described at the beginning of the proof. 
Using \Cref{1count}, we get the required result.
\end{proof}

We also note the following result which is easy to verify by studying how many occurrences of a binary word $p$ there are in a word obtained by adding a letter to $p$. 
Recall that $r_i$ denotes the number of runs of length $i$ in $p$.

\begin{lemma}\label{l+1}
For any $k \geq 2$, we have
\begin{equation*}
    B_{l + 1, p}(k) = r_{k - 1}.
\end{equation*}

\end{lemma}

\begin{example}
Let $p = \one\zero\zero\zero\one\one$. 
We obtain $w$ from $p$ by increasing the length of the last run from $2$ to $3$ by adding a $\one$, i.e., $w = \one\zero\zero\zero\one\one\one$. 
Clearly, $c_p(w) = \binom{3}{2} = 3$.
\end{example}

\subsection{The case $k=3$}

Let $p$ be a binary word of length $l$ that has $r$ runs, $r_i$ of which are of size $i$ for each $i \geq 1$. 
Also, let $r_{2, b}$ be the number of runs of size $2$ in $p$ that are `at the boundary', i.e., that are either the first or last run. 
For example, if $p = \one\one\one\zero\one\zero$, then $r_{2, b} = 0$, if $p = \one\zero\one\one$, then $r_{2, b} = 1$, and if $p = \one\one\zero\one\one\zero\zero$, then $r_{2, b} = 2$.

\begin{proposition}\label{3count}
For any $n \geq 3$, if $l \geq 3$, we have
\begin{equation*}
    B_{n, p}(3) = r_1 \binom{n - r - 1}{l - r + 1} + (r_2 - r_{2, b}) \binom{n - r - 1}{l - r} + r_{2,b} \binom{n - r}{l - r + 1}
\end{equation*}
and if $l < 3$, we have
\begin{equation*}
    B_{n, p}(3) =
    \begin{cases}
        \binom{n}{3} &\text{if $r = 1$}\\
        3(n - 3) &\text{if $r = 2$.}
    \end{cases}
\end{equation*}
\end{proposition}

\begin{proof}
For $l < 3$, we deal with the case $r = 2$; the case $r = 1$ is straightforward. 
Note that this means $p = \one\zero$ or $p = \zero\one$, which are trivially equivalent. 
We assume $p = \one\zero$. 
The result follows since it can be checked that words that contain exactly $3$ occurrences of $\one\zero$ are those of the form
\begin{equation*}
    \zero^a \one \zero \zero \zero \one^b \quad \text{or} \quad \zero^a \one \one \one \zero \one^b \quad \text{or} \quad \zero^a \one \zero \one \zero \one^b
\end{equation*}
for some $a, b \geq 0$.

We now consider the case $l \geq 3$. 
Just as in the previous proofs, we first describe methods to construct words that have $3$ occurrences of $p$. 
We then show that all such words are obtain from these methods.

\textbf{Method 3.1:}
Suppose $p_k$ is a run of size $1$ in $p$ and $w'$ is a word of length $n - 2$ such that it has exactly one occurrence of $p$ which is at indices $i_1 < i_2 < \cdots < i_l$. 
Adding two copies of the letter $p_k$ immediately after $w'_{i_k}$ results in a word $w$ of length $n$ such that $c_p(w) = 3$.

\begin{example}
Let $p = \one\one\zero\one\zero\zero$. 
Just as in \Cref{2countex}, we express an instance of the above method using a slot diagram. 
We choose $p_3 = \zero$ as the run of size $1$.
$$\one\ \one\ \textcolor{red}{\zero}\ \one\ \zero\ \zero
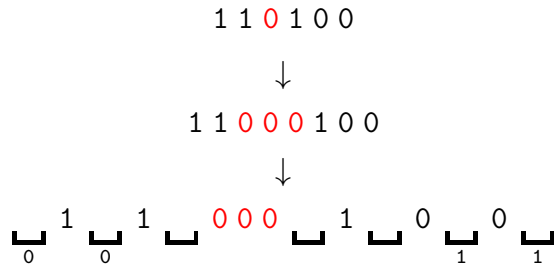
$$
$$\downarrow$$
$$\one\ \one\ \textcolor{red}{\zero\ \zero\ \zero}\ \one\ \zero\ \zero$$
$$\downarrow$$
\begin{equation*}
    \underbracket{\quad}_{\zero}\ \one\ \underbracket{\quad}_{\zero}\ \one\ \underbracket{\quad}\ \textcolor{red}{\zero\ \zero\ \zero}\ \underbracket{\quad}\ \one\ \underbracket{\quad}\ \zero\ \underbracket{\quad}_{\one}\ \zero\ \underbracket{\quad}_{\one}
\end{equation*}

Inserting letters into the appropriate slots, we can get the word $w = \one\zero\one\zero\zero\zero\one\zero\one\one\zero\one$. 
The fact that $c_p(w) = 3$ is illustrated below, where the two rows of dashes indicate the occurrences.
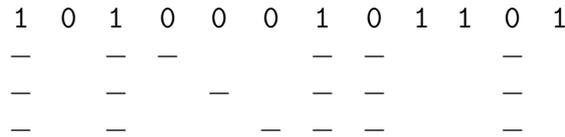
\setcounter{MaxMatrixCols}{20}
\begin{equation*}
    \begin{matrix}
        \one & \zero & \one & \zero & \zero & \zero & \one & \zero & \one & \one & \zero & \one\\
        - &  & - & - &  &  & - & - &  &  & - & \\
        - &  & - &  & - &  & - & - &  &  & - & \\
        - &  & - &  &  & - & - & - &  &  & - & \\
    \end{matrix}
\end{equation*}
\end{example}

\textbf{Method 3.2:}
Suppose $p_kp_{k + 1}$ is a run of size $2$ in $p$ and $w'$ is a word of length $n - 1$ that has exactly one occurrence of $p$ which is at indices $i_1 < i_2 < \ldots < i_l$. 
If $k = 1$, then the word $w$ of length $n$ obtained by adding the letter $p_1$ immediately after $w'_{i_2}$ satisfies $c_p(w) = 3$. 
If $k = l - 1$, then the word $w$ obtained by adding the letter $p_l$ immediately before $w'_{i_{l - 1}}$ satisfies $c_p(w) = 3$. 
Finally, suppose that $k$ is neither $1$ nor $l - 1$, i.e., the chosen run of size $2$ is not at the boundary. 
In this case, we require $i_{k + 1} = i_k + 1$, i.e., $w'$ has no letters between $w'_{i_k}$ and $w'_{i_{k + 1}}$. 
Then, the word $w$ obtained by adding the letter $p_k$ immediately after $w'_{i_{k + 1}}$ satisfies $c_p(w) = 3$.

\begin{example}
We illustrate the different cases for $p = \one\one\zero\zero\one\one$ using slot diagrams.
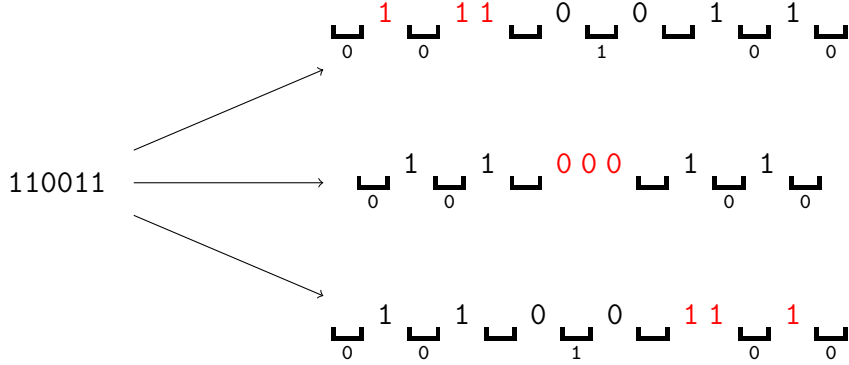
\begin{center}
    \begin{tikzpicture}
        \node[inner sep = 0.35cm] (p) at (0, 0) {$\one\one\zero\zero\one\one$};
        \node (1) at (7, 2) {$\underbracket{\quad}_{\zero}\ \textcolor{red}{\one}\ \underbracket{\quad}_{\zero}\ \textcolor{red}{\one\ \one\ }\underbracket{\quad}\ \zero\ \underbracket{\quad}_{\one}\ \zero\ \underbracket{\quad}\ \one\ \underbracket{\quad}_{\zero}\ \one\ \underbracket{\quad}_{\zero}$};
        \node (2) at (7, 0) {$\underbracket{\quad}_{\zero}\ \one\ \underbracket{\quad}_{\zero}\ \one\ \underbracket{\quad}\ \textcolor{red}{\zero\ \zero\ \zero}\ \underbracket{\quad}\ \one\ \underbracket{\quad}_{\zero}\ \one\ \underbracket{\quad}_{\zero}$};
        \node (3) at (7, -2) {$\underbracket{\quad}_{\zero}\ \one\ \underbracket{\quad}_{\zero}\ \one\ \underbracket{\quad}\ \zero\ \underbracket{\quad}_{\one}\ \zero\ \underbracket{\quad}\ \textcolor{red}{\one\ \one\ }\underbracket{\quad}_{\zero}\ \textcolor{red}{\one\ }\underbracket{\quad}_{\zero}$};
        
        \draw[->] (p) -- (3.5, 1.5);
        \draw[->] (p) -- (3.5, 0);
        \draw[->] (p) -- (3.5, -1.5);
    \end{tikzpicture}
\end{center}
\end{example}

From the above example, we can see that there is an extra slot for runs at the boundary because letters inserted into it cannot be used to obtain more occurrences of $p$.

We now show that any word $w$ such that $c_p(w) = 3$ can be obtained by the methods described above. 
Let $w$ be a binary word of length $n$ such that $c_p(w) = 3$. 
Let $\boldsymbol{i} : i_1 < i_2 < \cdots < i_l$, and similarly $\boldsymbol{h}$ and $\boldsymbol{j}$ be the indices where the occurrences of $p$ in $w$ appear. 
Also, suppose that $\boldsymbol{i} <_{\operatorname{lex}} \boldsymbol{h} <_{\operatorname{lex}} \boldsymbol{j}$ and that $k$ is such that $i_m = j_m$ for all $m < k$ and $i_k < j_k$. 
Here $<_{\operatorname{lex}}$ represents the lexicographic order.

\textbf{Case 1:} We first consider the case when $i_m = j_m$ for all $m > k$. 
Note that since $\boldsymbol{i} <_{\operatorname{lex}} \boldsymbol{h} <_{\operatorname{lex}} \boldsymbol{j}$, we have $h_m = i_m = j_m$ for all $m < k$. 
If $h_k = i_k$, we get a fourth occurrence of $p$ at $j_1 < j_2 < \cdots < j_k < h_{k + 1} < h_{k + 2} < \cdots < h_l$. 
For use in the proof of \Cref{4count}, we also note that this fourth occurrence is lexicographically larger than $\boldsymbol{j}$. 
We get a similar contradiction if $h_k = j_k$. 
Hence, we must have $i_k < h_k < j_k$. 
Using the same logic, we also get that $h_m = i_m = j_m$ for all $m > k$, $p_k$ forms a run of size $1$ in $p$, and $i_k + 1 = h_k = j_k - 1$.

Just as in the proof of \Cref{2count}, we see that deleting $w_{h_k}$ and $w_{j_k}$ results in a word of length $n - 2$ having exactly one occurrence of $p$. 
This shows that $w$ can be constructed using \textbf{Method 3.1}. 
This gives us the first term in the expression for $B_{n, p}(3)$.

\textbf{Case 2:} Next, we assume that there exists some $m > k$ such that $i_m \neq j_m$. 
We first show that $i_{k + 1} = j_k$. 
If $i_{k + 1} > j_k$, then using the occurrences given by $i_1 < \cdots < i_k < j_{k + 1} < \cdots < j_l$ and $j_1 < \cdots < j_k < i_{k + 1} < \cdots < i_l$, we get at least four occurrences of $p$ in $w$. 
Suppose that $i_{k + 1} < j_k$. 
Just as before, the indices $i_1 < i_2 < \cdots < i_k < j_{k + 1} < \cdots < j_l$ give a third occurrence of $p$ in $w$. 
For $i_1 < i_2 < \cdots < i_{k + 1} < j_{k + 2} < \cdots < j_l$ to not be a fourth occurrence, we must have $i_m = j_m$ for all $m > k + 1$. 
Since $l > 2$, we must have either $k > 1$ or $k + 1 < l$. 
We deal with the case $k > 1$ since the other can be dealt with analogously. 
If $p_k = p_{k + 1}$, then $i_1 < \cdots < i_{k - 1} < i_{k + 1} < j_k < \cdots < j_l$ gives a fourth occurrence of $p$ in $w$. 
Hence $p_k \neq p_{k + 1}$ and we show that $p_{k - 1}$ can be neither of them. 
If $p_{k - 1} = p_k$, then $i_1 < \cdots < i_{k - 2} < i_k < j_k < \cdots < j_l$ gives a fourth occurrence of $p$ in $w$. 
We get a similar contradiction if $p_{k - 1} = p_{k + 1}$. 
Hence, we must have $i_{k + 1} = j_k$. 
In particular, this shows that $p_k = p_{k + 1}$.

Note that we have the third occurrence $\boldsymbol{h} = i_1 < \cdots < i_k < j_{k + 1} < \cdots < j_l$ and for there to be no other occurrences of $p$, we must have $i_m = j_m$ for all $m > k + 1$. 
Using similar ideas as before, we can show that $p_kp_{k + 1}$ forms a run of size $2$ in $p$. 
If this run is neither the first nor last run in $p$, then we must have $i_k + 1 = i_{k + 1} = j_k = j_{k + 1} - 1$ and deleting $w_{j_k}$ and $w_{j_{k + 1}}$ gives a word of length $n - 2$ that has exactly one occurrence of $p'$, where $p'$ is the word obtained by deleting $p_{k + 1}$ from $p$. 
If $k = 1$, then $j_{k + 1} = j_k + 1$ and we must have letters not equal to $p_k$ between $w_{i_k}$ and $w_{j_k}$. 
Also, in this case, deleting the letter $w_{j_{k + 1}}$ results in a word of length $n - 1$ that has exactly one occurrence of $p$. 
Analogous statements hold if $k = l - 1$.

This shows that $w$ can be constructed using \textbf{Method 3.2}. 
This gives us the second and third term in the expression for $B_{n, p}(3)$.
\end{proof}

\subsection{The case $k=4$}

Let $p$ be a binary word of length $l$ that has $r$ runs, $r_i$ of which are of size $i$ for each $i \geq 1$. 
Let $r_{3, b}$ be the number of runs of size $3$ in $p$ that are either the first or last run. 
Also, set $r_{(1, 2), b}$ to be $2$ if the first two runs in both $p$ and $p^r$ are of sizes $1$ and $2$ respectively, $1$ if this happens for exactly one of $p$ or $p^r$, and $0$ otherwise. 
For example, if $p = \one\zero\one\one\zero$, then $r_{(1, 2), b} = 1$ since the first two runs of $p^r$ are of sizes $1$ and $2$ respectively.

\begin{proposition}\label{4count}
For any $n \geq 5$, if $l \geq 4$, we have
\begin{align*}
    B_{n, p}(4) &= r_1 \binom{n - r - 2}{l - r + 1} + \binom{r_1}{2}\binom{n - r - 1}{l - r + 1} + (r_3 - r_{3, b}) \binom{n - r - 1}{l - r}\\
    &+ r_{3,b} \binom{n - r}{l - r + 1} + r_{(1, 2), b} \binom{n - r - 2}{l - r}
\end{align*}
and otherwise, we have
\begin{equation*}
    B_{n, p}(4) =
    \begin{cases}
        0 &\text{if $l = 2, r = 1$}\\
        \binom{n}{4} &\text{if $l = 1\text{ or }3, r = 1$}\\
        (n - 4)^2 &\text{if $l = 3, r = 2$}\\
        5n - 19 &\text{if $l = 2, r = 2$}\\
        7n - 31 &\text{if $l = 3, r = 3$.}\\
    \end{cases}
\end{equation*}
\end{proposition}

\begin{proof}
We first consider $l < 4$. 
The results for $r = 1$ are straightforward. 
For the other cases, we state, without proof, the form of the words that contain exactly $4$ occurrences of $p$. 
The expressions for $B_{n, p}(4)$ can be derived from this description. 

If $l = 3$ and $r = 2$, we can assume that $p = \one\zero\zero$ using trivial equivalences. 
Any word that contains exactly $4$ occurrences of $p$ must be of the form
\begin{equation*}
    \zero^a \one \zero \one \zero \one^b \zero \one^c \quad \text{or} \quad \zero^a \one \one \one \one \zero \one^b \zero \one^c
\end{equation*}
for some $a, b, c \geq 0$.

If $l = r = 2$, we can assume $p = \one \zero$. 
Any word that contains exactly $4$ occurrences of $p$ must be of the form
\begin{equation*}
    \zero^a \one \one \zero \zero \one^b \quad \text{or}\quad \zero^a \one \zero \zero \one \zero \one^b\quad \text{or}\quad \zero^a \one \zero \zero \zero \zero \one^b \quad \text{or}\quad  \zero^a \one \zero \one \one \zero \one^b \quad \text{or}\quad  \zero^a \one \one \one \one \zero \one^b
\end{equation*}
for some $a, b \geq 0$.

If $l = r = 3$, we can assume $p = \one \zero \one$. 
Any word that contains exactly $4$ occurrences of $p$ must be of the form
\begin{gather*}
    \zero^a \one \zero \zero \one \one \zero^b \quad\text{or}\quad \zero^a \one \zero \one \zero \one \zero^b \quad\text{or}\quad \zero^a \one \one \zero \zero \one \zero^b \quad\text{or}\quad \zero^a \one \one \zero \one \one \zero^b \quad\text{or}\\
    \zero^a \one \zero \zero \zero \zero \one \zero^b \quad\text{or}\quad
    \zero^a \one \zero \one \one \one \one \zero^b \quad\text{or}\quad
    \zero^a \one \one \one \one \zero \one \zero^b
\end{gather*}
for some $a, b \geq 0$.

We now consider the case $l \geq 4$. 
Just as in the previous proofs, we first describe methods to construct words that have $3$ occurrences of $p$. 
We then show that all such words are obtain from these methods.

\textbf{Method 4.1:}
This method is very similar to \textbf{Method 3.1} in \Cref{3count}. 
We illustrate it in the example below.

\begin{example}
Let $p = \one\one\zero\one\zero\zero$. 
We express an instance of this method using a slot diagram. 
We choose $p_3 = \zero$ as the run of size $1$.
$$\one\ \one\ \textcolor{red}{\zero}\ \one\ \zero\ \zero
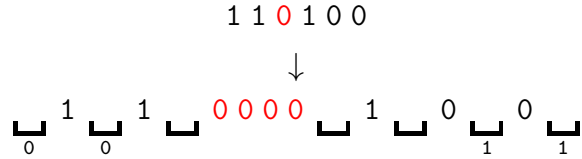
$$
$$\downarrow$$
\begin{equation*}
    \underbracket{\quad}_{\zero}\ \one\ \underbracket{\quad}_{\zero}\ \one\ \underbracket{\quad}\ \textcolor{red}{\zero\ \zero\ \zero\ \zero}\ \underbracket{\quad}\ \one\ \underbracket{\quad}\ \zero\ \underbracket{\quad}_{\one}\ \zero\ \underbracket{\quad}_{\one}
\end{equation*}
Inserting appropriate letters into slots results in words that have $4$ occurrences of $p$.
\end{example}

\textbf{Method 4.2:}
Suppose $p_k$ and $p_s$ for $1 \leq k < s \leq l$ are runs of size $1$ in $p$. 
Let $w'$ be a word that has exactly one occurrence of $p$ which is at indices $i_1 < i_2 < \cdots < i_l$. 
Construct a word $w$ by adding a copy of the letter $p_k$ immediately after $w'_{i_k}$ and a copy of the letter $p_s$ immediately after $w'_{i_s}$. 
Then $w$ satisfies $c_p(w) = 4$.

\begin{example}
Let $p = \one\one\zero\one\zero\zero\one$. 
We express an instance of this method using a slot diagram. 
We choose $p_3 = \zero$ and $p_7 = \one$ as the runs of size $1$.
$$\one\ \one\ \textcolor{red}{\zero}\ \one\ \zero\ \zero\ \textcolor{red}{\one}
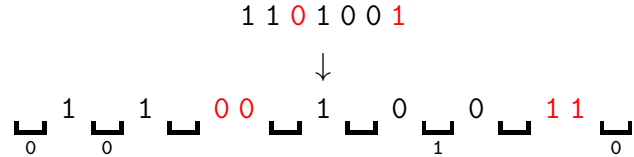
$$
$$\downarrow$$
\begin{equation*}
    \underbracket{\quad}_{\zero}\ \one\ \underbracket{\quad}_{\zero}\ \one\ \underbracket{\quad}\ \textcolor{red}{\zero\ \zero}\ \underbracket{\quad}\ \one\ \underbracket{\quad}\ \zero\ \underbracket{\quad}_{\one}\ \zero\ \underbracket{\quad}\ \textcolor{red}{\one\ \one}\ \underbracket{\quad}_{\zero}
\end{equation*}
\end{example}

\textbf{Method 4.3:}
This method is very similar to \textbf{Method 3.2} in \Cref{3count}. 
We illustrate it in the example below.

\begin{example}
We illustrate the method for boundary and non-boundary runs of size $3$ for $p = \one\one\one\zero\zero\zero\one\one\one$ using slot diagrams.
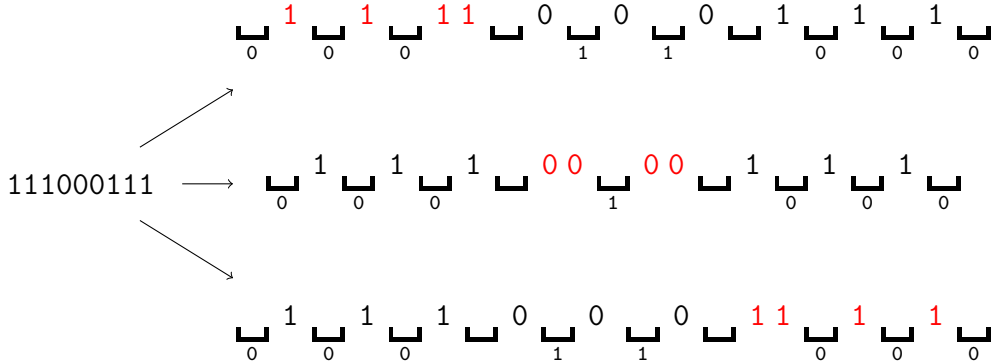
\begin{center}
    \begin{tikzpicture}
        \node[inner sep = 0.35cm] (p) at (0, 0) {$\one\one\one\zero\zero\zero\one\one\one$};
        \node (1) at (7, 2) {$\underbracket{\quad}_{\zero}\ \textcolor{red}{\one}\ \underbracket{\quad}_{\zero}\ \textcolor{red}{\one}\ \underbracket{\quad}_{\zero}\ \textcolor{red}{\one\ \one\ }\underbracket{\quad}\ \zero\ \underbracket{\quad}_{\one}\ \zero\ \ \underbracket{\quad}_{\one}\ \zero\ \underbracket{\quad}\ \one\ \underbracket{\quad}_{\zero}\ \one\ \underbracket{\quad}_{\zero}\ \one\ \underbracket{\quad}_{\zero}$};
        \node (2) at (7, 0) {$\underbracket{\quad}_{\zero}\ \one\ \underbracket{\quad}_{\zero}\ \one\ \underbracket{\quad}_{\zero}\ \one\ \underbracket{\quad}\ \textcolor{red}{\zero\ \zero}\ \underbracket{\quad}_{\one}\ \textcolor{red}{\zero\ \zero}\ \underbracket{\quad}\ \one\ \underbracket{\quad}_{\zero}\ \one\ \underbracket{\quad}_{\zero}\ \one\ \underbracket{\quad}_{\zero}$};
        \node (3) at (7, -2) {$\underbracket{\quad}_{\zero}\ \one\ \underbracket{\quad}_{\zero}\ \one\ \underbracket{\quad}_{\zero}\ \one\ \underbracket{\quad}\ \zero\ \underbracket{\quad}_{\one}\ \zero\ \ \underbracket{\quad}_{\one}\ \zero\ \underbracket{\quad}\ \textcolor{red}{\one\ \one\ }\underbracket{\quad}_{\zero}\ \textcolor{red}{\one\ }\underbracket{\quad}_{\zero}\ \textcolor{red}{\one\ }\underbracket{\quad}_{\zero}$};
        
        \draw[->] (p) -- (2, 1.25);
        \draw[->] (p) -- (2, 0);
        \draw[->] (p) -- (2, -1.25);
    \end{tikzpicture}
\end{center}    
\end{example}

\textbf{Method 4.4:}
Suppose the first run of $p$ is of size $1$ and the second is of size $2$. 
Without loss of generality, let the first letter of $p$ be $\one$. 
Let $p'$ be the word obtained by deleting the first $3$ letters of $p$. 
Suppose that $w'$ is a word that has exactly one occurrence of $p'$ which is at indices $i_1 < i_2 < \cdots < i_{l - 3}$. 
Then the word $w$ obtained by inserting the letters $\one\zero\one\zero\zero$ in $w'$ immediately before $w'_{i_1}$ satisfies $c_p(w) = 4$. 
A similar method can be used to construct words with $4$ occurrences of $p$ if the first two runs of $p^r$ are of sizes $1$ and $2$ respectively.

\begin{example}
    We illustrate an instance of both cases of this method using $p = \one\zero\zero\one\one\zero$. 
    The slot diagram corresponding to using the first two runs is as follows. 

    $$\textcolor{red}{\one\ \zero\ \zero}\ \one\ \one\ \zero
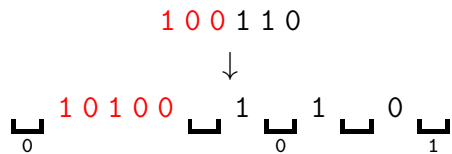
$$
    $$\downarrow$$
    \begin{equation*}
    \underbracket{\quad}_{\zero}\ \textcolor{red}{\one\ \zero\ \one\ \zero\ \zero}\ \underbracket{\quad}\ \one\ \underbracket{\quad}_{\zero}\ \one\ \underbracket{\quad}\ \zero\ \underbracket{\quad}_{\one}
    \end{equation*}
    The one corresponding to using the last two runs is as follows. 
    $$\one\ \zero\ \zero\ \textcolor{red}{\one\ \one\ \zero}$$
    $$\downarrow$$
    \begin{equation*}
        \underbracket{\quad}_{\zero}\ \one\ \underbracket{\quad}\ \zero\ \underbracket{\quad}_{\one}\ \zero\ \underbracket{\quad}\ \textcolor{red}{\one\ \one\ \zero\ \one\ \zero}\ \underbracket{\quad}_{\one}
    \end{equation*}
\end{example}

We now show that any word $w$ such that $c_p(w) = 4$ can be obtained by the methods described above. 
Let $w$ be a binary word of length $n$ such that $c_p(w) = 4$. 
Let $\boldsymbol{i}: i_1 < i_2 < \cdots < i_l$ be the lexicographically smallest occurrence of $p$ in $w$ and $\boldsymbol{j}$ be the largest. 
Let $k \in [l]$ be the smallest such that $i_k < j_k$. 
Note that $i_m = j_m$ for all $m < k$.

\textbf{Case 1:}
We first consider the case when $i_m = j_m$ for all $m > k$. 
This case is very similar to the analogous case in \Cref{3count}. 
This case corresponds to \textbf{Method 4.1} and contributes to the first term in the expression for $B_{n, p}(4)$.

\textbf{Case 2:}
Next, we assume that there exists some $m > k$ such that $i_m \neq j_m$. 
Let $s \in [k + 1, l]$ be the smallest such $m$.

\textbf{Case 2(a):}
Suppose $j_k < i_s$. 
We show that this case corresponds to \textbf{Method 4.2}. 
Note that the four occurrences of $p$ in $w$ are as follows:
\begin{itemize}
    \item $\boldsymbol{i}: i_1 < i_2 < \cdots < i_l$
    \item $i_1 < i_2 < \cdots < i_{s - 1} < j_s < j_{s + 1} < \cdots < j_l$
    \item $j_1 < j_2 < \cdots < j_{s - 1} < i_s < i_{s + 1} < \cdots < i_l$
    \item $\boldsymbol{j}: j_1 < j_2 < \cdots < j_l$
\end{itemize}
For there to be no other occurrences of $p$ in $w$, we must have the following:
\begin{itemize}
    \item $i_m = j_m$ for all $m \neq k, s$.
    \item $j_k = i_k + 1$ and $j_s = i_s + 1$.
    \item $p_k$ and $p_s$ form runs of size $1$ in $p$.
    \item Deleting $w_{j_k}$ and $w_{j_s}$ results in a word having exactly one occurrence of $p$.
\end{itemize}

Hence, $w$ can be constructed using \textbf{Method 4.2} and this case corresponds to the second term in the expression for $B_{n, p}(4)$.

\textbf{Case 2(b):}
Suppose $j_k = i_s$. 
Note that in this case we must have $s = k + 1$ and hence $p_k = p_{k + 1}$. 
First we consider the case $p_{k + 1} = p_{k + 2}$ and show that this corresponds to \textbf{Method 4.3}. 
We obtain four occurrences of $p$ in $w$ by taking the indices $i_1 < i_2 < \cdots < i_{k - 1}$ for the first $k - 1$ letters of $p$, any three of the indices $i_k, j_k, j_{k + 1}, j_{k + 2}$ for the next $3$ letters in $p$, and finally the indices $j_{k + 3} < \cdots < j_l$ for the remaining letters of $p$. 
For there to be no other occurrences of $p$, the letters $p_kp_{k + 1}p_{k + 2}$ must form a run of size $3$. 
Hence, $w$ can be constructed using \textbf{Method 4.3} and this case corresponds to the third and fourth term in the expression for $B_{n, p}(4)$.

Next, we assume that $p_{k + 1} \neq p_{k + 2}$. 
Note that if $i_{k + 2} > j_{k + 1}$, then we obtain more than $4$ occurrences of $p$ in $w$. 
Hence, $i_{k + 2} < j_{k + 1}$. 
In fact, the indices $i_{k + 1}, i_{k + 2}, j_{k + 1}, j_{k + 2}$ must be consecutive. 
We can also check that we must have $k + 2 = l$. 
Since $l \geq 4$, $k \neq 1$ and we can check that the letter $p_{k - 1} \neq p_k$ and this implies that $i_k + 1 = i_{k + 1}$. 
Also, deleting the letters $w_{i_k}, w_{i_k + 1}, \ldots, w_{i_k + 4}$ results in a word that has exactly one occurrence of $p'$. 
Here $p'$ is the word obtained from $p$ by deleting its last three letters. 
This shows that $w$ can be constructed using \textbf{Method 4.4} using the last two runs of $p$. 
Hence, this case accounts for the last term in the expression for $B_{n, p}(4)$ corresponding to when $p^r$ has first two runs of sizes $1$ and $2$ respectively.

\textbf{Case 2(c):}
Finally we assume that $j_k > i_s$. 
In this case as well, we must have $s = k + 1$. 
If $p_k = p_{k + 1}$, we obtain more than $4$ occurrences of $p$ in $w$ by using the indices $i_1 < i_2 < \cdots < i_{k - 1}$ for the first $k - 1$ letters of $p$, any two of $i_k, i_{k + 1}, j_k, j_{k + 1}$ for the next two letters, and $j_{k + 2} < \cdots < j_l$ for the remaining letters. 
Hence $p_k \neq p_{k + 1}$. 
We can also show that $p_{k + 2} = p_{k + 1}$ using the fact that $l \geq 4$. 
Similar to the previous case, we get that $k = 1$, the indices $i_k, i_{k + 1}, j_k, j_{k + 1}, j_{k + 2}$ are consecutive, and that deleting $w_{i_k}, w_{i_k + 1}, \ldots , w_{i_k + 4}$ results in a word with one occurrence of $p'$. 
Here $p'$ is the word obtained by removing the first three letters from $p$. 
This shows that $w$ can be constructed using \textbf{Method 4.4} using the first two runs of $p$. 
Hence, this case accounts for the last term in the expression for $B_{n, p}(4)$ corresponding to when $p$ has first two runs of sizes $1$ and $2$ respectively.
\end{proof}

\section{Maximum occurrences and internal zeroes}\label{maxoccandintzerosec}

Given $n \geq 0$ and a binary word $p$, we set $M_{n, p}$ to be maximum possible number of occurrences of $p$ in a binary word of length $n$. 
Hence,
\begin{align*}
    M_{n, p} &= \operatorname{max}\{c_p(w) \mid w \in \{\texttt{0}, \texttt{1}\}^n\}\\[0.2cm]
    &= \operatorname{max}\{k \mid B_{n, p}(k) \neq 0\}.
\end{align*}
A binary word $w$ of length $n$ is said to be \emph{$p$-optimal} if $c_p(w) = M_{n, p}$.

\begin{proposition}\label{maxform}
Let \textup{$p = \texttt{1}^i \texttt{0}^j \texttt{1}^k$} for some $i, k \geq 0$ and $j \geq 1$. 
For any $n \geq i + j + k$, we have
\begin{equation*}
    M_{n, p} = \operatorname{max}\left\{\binom{a}{i}\binom{b}{j}\binom{c}{k} \mid a + b + c = n\right\}.
\end{equation*}
\end{proposition}

\begin{proof}
Note that the number of occurrences of $p$ in the binary word $\texttt{1}^a\texttt{0}^b\texttt{1}^c$ is $\binom{a}{i}\binom{b}{j}\binom{c}{k}$. 
We prove the result by showing that for any binary word $w$ of length $n$, there exists a binary word $w' = \texttt{1}^a\texttt{0}^b\texttt{1}^c$ of length $n$ such that $c_p(w') \geq c_p(w)$.

Let $w$ be a binary word of length $n$. 
Each occurrence of $\texttt{0}^j$ in $w$ contributes $\binom{a}{i}\binom{c}{k}$ to the number of occurrences of $p$ where
\begin{itemize}
    \item $a$ is the number of $\texttt{1}$s in $w$ before the first letter in the occurrence of $\texttt{0}^j$ and
    \item $b$ is the number of $\one$s in $w$ after the last letter in the occurrence.
\end{itemize}
Consider an occurrence of $\texttt{0}^j$ that makes the maximum contribution among all occurrences of $\texttt{0}^j$ in $w$. 
Suppose its contribution is $\binom{a}{i}\binom{c}{k}$ when written in the form mentioned above.

Let $w'$ be the binary word $\texttt{1}^a\texttt{0}^b\texttt{1}^c$ where $b = n - a - c$. 
Note that there are at least as many occurrences of $\texttt{0}^j$ in $w'$ as there are in $w$ since it has at least as many $\texttt{0}$s. 
Also, each occurrence of $\texttt{0}^j$ in $w'$ contributes $\binom{a}{i}\binom{c}{k}$ to the number of occurrences of $p$. 
This shows that $c_p(w') \geq c_p(w)$ and hence proves the required result.
\end{proof}

Any binary word with at most $3$ runs is trivially equivalent to a word of the form mentioned in the above proposition. 
This gives us the following result, which can be viewed as an analogue of \cite[Theorem 3.1]{freqseq} for binary words.

\begin{corollary}\label{layeranalogue}
Let $p$ be a binary word of length $l$ that has $r$ runs where $r \leq 3$. 
For any $n \geq l$, there exists a $p$-optimal binary word of length $n$ that has $r$ runs.
\end{corollary}

In fact, if $p$ has $2$ runs, then these are the only $p$-optimal binary words. 
We leave the proof of this as an easy exercise.

\begin{result}\label{2runoptform}
If $p$ is a binary word with $2$ runs, then any $p$-optimal binary word also has $2$ runs.
\end{result}

This result does not extend to words with $3$ runs. 
For example, it can be checked that for $p = \one\one\zero\one$, the word $\one\one\one\zero\one\zero\one$ is $p$-optimal. 
However, an analogue of the above result does hold for certain binary words with $3$ runs (see \Cref{3runallopt3run}).

We note that \Cref{layeranalogue} does not hold for binary words with $4$ or more runs. 
In particular, we have the following: 
Let $A_l$ denote the alternating binary word of length $l$ starting with $\texttt{1}$, i.e., $A_l$ starts with $\texttt{1}$ and has all runs of size $1$. 
For example, $A_4 = \texttt{1010}$, $A_5 = \texttt{10101}$. 
It can be shown that if $l \geq 4$, the only $A_l$-optimal binary word of length $l + 2$ is $A_{l + 2}$.

\begin{definition}
A binary word $p$ is said to have an \emph{internal zero} at $n$ if the sequence $(B_{n, p}(k))_{k \geq 0}$ has an internal zero. 
This means that there exist $0 \leq k_1 < k_2 < k_3$ such that $B_{n, p}(k_1), B_{n, p}(k_3) \neq 0$ but $B_{n, p}(k_2) = 0$.
\end{definition}

Note that if $p$ is of length $l$, then it cannot have an internal zero at any $n \leq l$. 
It is easy to see that for any $n \geq l + 1$, we have $B_{n, p}(0), B_{n, p}(1) \neq 0$ (see \Cref{0count,1count}). 
Hence, to check if $p$ has an internal zero at $n$, we only need to find $2 \leq k_2 < k_3$ such that $B_{n, p}(k_2) = 0$ and $B_{n, p}(k_3) \neq 0$.

\begin{proposition}
Let $p$ be a binary word of length $l$ all of whose runs are of size at least $2$. 
Then $p$ has an internal zero at $n$ for all $n \geq l + 1$.
\end{proposition}

\begin{proof}
Let $n \geq l + 1$. 
From \Cref{2count}, we see that $B_{n, p}(2) = 0$ since $p$ has no run of size $1$. 
If the first letter of $p$ is $\texttt{1}$ and the first run is of size $i$, then the word $\texttt{1}^{n - l}p$ has $k = \binom{n - l + i}{i}$ occurrences of $p$. 
Since $i \geq 2$, we have $k \geq 3$, and hence $B_{n, p}(k) \neq 0$. 
This shows that $p$ has an internal zero at $n$.
\end{proof}

As an immediate consequence of \Cref{l+1}, we have the following.

\begin{result}
Let $p$ be a binary word of length $l$ with maximum run size $i$. 
Then $p$ does not have an internal zero at $l + 1$ if and only if $p$ has a run of size $j$ for all $j \in [i]$.
\end{result}

In the following results, we characterize the values at which binary words with at most $3$ runs have internal zeroes. We start with the following result for binary words with just one run, which is easy to verify.

\begin{result}
Let $p$ be a binary word with $1$ run. 
If \textup{$p = \texttt{0}\text{ or }\texttt{1}$} then it does not have an internal zero at any $n \geq 0$. 
Otherwise, if $p$ is of length $l \geq 2$, it has an internal zero at all $n \geq l + 1$.
\end{result}

In fact, similar results hold for words with $2$ or $3$ runs as stated in the following theorem, where we say that a binary word is \emph{alternating} if all its runs are of size $1$.

\begin{theorem}\label{IZgenres}
Let $p$ be a binary word of length $l$ with at most $3$ runs.
\begin{itemize}
    \item If $p$ is alternating, then $p$ does not have an internal zero at any $n \geq 7$.
    \item If $p$ is not alternating, then $p$ has an internal zero at all $n \geq l + 3$. 
    In fact, we have $B_{n, p}(M_{n, p} - 1) = 0$.
\end{itemize}
\end{theorem}

Moreover, the only non-alternating words $p$ that do not satisfy $B_{n, p}(M_{n, p} - 1) = 0$ for $n = l + 2$ are those trivially equivalent to $\one\one\zero\one$. 
The propositions that follow prove the above theorem.

\begin{proposition}
The binary word \textup{$\texttt{10}$} does not have an internal zero at any $n \geq 0$.
\end{proposition}

\begin{proof}
A binary word that has at least one occurrence of $\texttt{10}$ must have an occurrence at consecutive indices. 
If not, then the word is of the form $\texttt{0}^a\texttt{1}^b$ for some $a, b \geq 0$ and hence has no occurrences of $\texttt{10}$. 
Let $w$ be a binary word of length $n$ with $w_i = \texttt{1}$ and $w_{i + 1} = \texttt{0}$ for some $i \in [n - 1]$. 
It can be checked that the word obtained by swapping $w_i$ and $w_{i + 1}$ has exactly one less occurrence of $\texttt{10}$. 
This can be used to show that $\texttt{10}$ does not have any internal zeroes by starting with a $\one\zero$-optimal word $w$ and repeatedly swapping consecutive terms of the form $\texttt{10}$.

\begin{example}
    Starting with the word $\one\one\zero\zero$, which is $\one\zero$-optimal, an instance of the above procedure is as follows. 
    \begin{equation*}
        \one\one\zero\zero \rightarrow \one\zero\one\zero \rightarrow \zero\one\one\zero \rightarrow \zero\one\zero\one \rightarrow \zero\zero\one\one
    \end{equation*}
    It can be checked that these words have $4, 3, 2, 1, 0$ occurrences of $\one\zero$ respectively.
\end{example}
\end{proof}

\begin{proposition}\label{2runsmaxminus1}
Let $p$ be a binary word with $2$ runs that is of length $l \geq 3$. 
For all $n \geq l + 2$, we have $B_{n, p}(M_{n, p} - 1) = 0$. 
In particular, $p$ has an internal zero at $n$ for all $n \geq l + 2$.
\end{proposition}

\begin{proof}
Using trivial equivalences, we can assume that $p = \texttt{1}^i\texttt{0}^j$ for some $i \geq j$ and $i \geq 2$. 
Suppose $n \geq l + 2$. 
We first show that no word of the form $\texttt{1}^k\texttt{0}^{n - k}$ for $k \in [i, n - j]$ can have exactly $M_{n, p} - 1$ occurrences of $p$. 
Set $a_k = \binom{k}{i}\binom{n - k}{j}$ for all $k \in [i, n - j]$. 
Note that the sequence $a_i, a_{i + 1}, \ldots, a_{n - j}$ is unimodal since it is log-concave. 
By \Cref{maxform}, we have $M_{n, p} = \operatorname{max}\{a_k \mid k \in [i, n - j]\}$. 
Hence, we can prove our result by showing that $|a_k - a_{k + 1}| \neq 1$ for any $k \in [i, n - j - 1]$. 
This follows since $|a_k - a_{k + 1}| = 1$ implies that
\begin{equation*}
    \frac{k - i + 1}{\dbinom{k}{k - i}} \cdot \frac{n - k - j}{\dbinom{n - k - 1}{n - k - j - 1}} = \pm(ni - ki - kj - j).
\end{equation*}
Using the fact that for any $a \geq b + 1$, $\binom{a}{b} \geq b + 1$ with equality only when $a = b + 1$ or $b = 0$, we see that the left-hand side is an integer only in the following cases:
\begin{itemize}
    \item When $k = i = n - j - 1$. 
    In this case, we get $n = i + j + 1 = l + 1$, which contradicts our choice of $n$.
    \item When $k = i$ and $j = 1$. 
    In this case, the left-hand side is $1$. 
    But we also have
    \begin{equation*}
        ni - ki - kj - j = i(n - i - 1) - 1 \geq 3
    \end{equation*}
    which means the right-hand side can never be $1$.
\end{itemize}
This proves that no word of the form $\one^k\zero^{n - k}$ can have $M_{n, p} - 1$ occurrences of $p$.

We now show that any word that contains $p$ and has at least one occurrence of $\texttt{01}$ cannot have exactly $M_{n, p} - 1$ occurrences of $p$. 
Suppose to the contrary that such a word $v$ exists. 
Using \Cref{2runoptform}, we can assume $v$ has exactly one occurrence of $\zero\one$ by swapping adjacent occurrences of $\texttt{01}$. 
Hence, we have
\begin{equation*}
    v = \texttt{1}^{k - 1} \texttt{0} \texttt{1} \texttt{0}^{n - k - 1}
\end{equation*}
for some $k \in [i, n - j]$. 
Setting $w = \texttt{1}^{k}\texttt{0}^{n - k}$, we get
\begin{equation*}
    c_p(v) = c_p(w) + \binom{k - 1}{i - 1}\binom{n - k - 1}{j - 1}.
\end{equation*}
Since $w$ has $2$ runs and $c_p(v) = M_{n, p} - 1$, we get $c_p(w) = M_{n, p}$. 
If $j \geq 2$, this gives $k = i$ and $n - k = j$ and hence $n = i + j = l$, which is a contradiction. 
If $j = 1$, then we still get $k = i$. 
But it can be checked that in this case, $\texttt{1}^{i + 1}\texttt{0}^{n - i - 1}$ has more occurrences of $p$ than $w$ and hence contradicts the definition of $M_{n, p}$.
\end{proof}

\begin{proposition}
The binary word $\one\zero\one$ does not have an internal zero at any $n \geq 0$ except $n = 6$.
\end{proposition}

\begin{proof}
We prove this result by induction on $n$. 
Using Sage \cite{Sage}, it can be checked that the result holds for $n \leq 11$. 
In particular, $B_{6, \one\zero\one}(5) = 0$ and $B_{6, \one\zero\one}(6) = 9$ and hence $\one\zero\one$ has an internal zero at $6$. 
We assume $n \geq 12$. 
Using the induction hypothesis, we can obtain a binary word of length $n - 1$ that has $k$ occurrences of $\one\zero\one$ for any $k \in [0, M_{n - 1, p}]$. 
Appending a $\zero$ to the end of this word gives a binary word of length $n$ that has $k$ occurrences of $\one\zero\one$. 
Hence, to prove the result we have to show that for any $k \in [M_{n - 1, p} + 1, M_{n, p}]$, there is a binary word of length $n$ that has exactly $k$ occurrences of $\one\zero\one$.

Using \Cref{maxform}, we can show that there is a $\one\zero\one$-optimal word of length $n$ of the form
\begin{equation*}
    \one^m \zero^m \one^m \quad \text{or} \quad \one^m \zero^m \one^{m + 1} \quad \text{or} \quad \one^m \zero^{m + 1} \one^{m + 1}
\end{equation*}
for some $m \geq 4$ (since $n \geq 12$). 
We consider the case when $n = 3m$, that is, when $w = \one^m \zero^m \one^m$ is $\one\zero\one$-optimal and $M_{n, \one\zero\one} = m^3$. 
The other cases can be dealt with similarly. 
Note that in this case, we have $M_{n - 1, \one\zero\one} = m^3 - m^2$ since $\one^{m - 1}\zero^m \one^m$ is $\one\zero\one$-optimal. 
We prove the result by showing that, given any $k \in [m^2]$, we can shift $\zero$s in $w$ so that the resulting word has $m^3 - k$ occurrences of $\one\zero\one$.

Consider the spaces between the $\one$s in the third run of $w$ as slots. 
Hence, there are $m$ slots where the $i^{th}$ slot is the space immediately after the $i^{th}$ $\one$ in the third run of $w$. 
We claim that moving a $\zero$ from the second run to the $i^{th}$ slot reduces the number of occurrences of $\one\zero\one$ by $i^2$. 
This is because the $\zero$ originally contributed $m^2$ occurrences of $\one\zero\one$ but now contributes $(m + i)(m - i) = m^2 - i^2$ occurrences. 
Using the same idea, if $j$ $\zero$s from the second run are moved to the (not necessarily distinct) slots $i_1, i_2, \ldots, i_j$, then the number of occurrences of $\one\zero\one$ decreases by $i_1^2 + i_2^2 + \cdots + i_j^2$.

\begin{example}
    For $m = 4$, suppose we shift $\zero$s in $w$ into slots as shown below. 
    
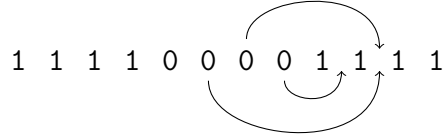
\begin{center}
        \begin{tikzpicture}
            \node at (0, 0) {\one};
            \node at (0.5, 0) {\one};
            \node at (1, 0) {\one};
            \node at (1.5, 0) {\one};
            
            \node at (2, 0) {\zero};
            \node (1) at (2.5, 0) {\zero};
            \node (2) at (3, 0) {\zero};
            \node (3) at (3.5, 0) {\zero};
            
            \node at (4, 0) {\one};
            \node (s1) at (4.25, 0) {};
            \node at (4.5, 0) {\one};
            \node (s2) at (4.75, 0) {};
            \node at (5, 0) {\one};
            \node at (5.5, 0) {\one};

            \draw [->] (1.south)..controls +(down:8mm) and +(down:12mm)..(s2.south);

            \draw [->] (2.north)..controls +(up:6mm) and +(up:9mm)..(s2.north);

            \draw [->] (3.south)..controls +(down:3mm) and +(down:5mm)..(s1.south);
        \end{tikzpicture}
    \end{center}
    This results in the word $\one\one\one\one\zero\one\zero\one\zero\zero\one\one$ which has $4^3 - 1^2 - 2^2 - 2^2 = 55$ occurrences of $\one\zero\one$.
\end{example}

Since $m \geq 4$, the result now follows from Lagrange's four-square theorem \cite[Theorem 11-3]{lagthm} which states that any positive integer can be written as the sum of $4$ integral squares.
\end{proof}

\begin{proposition}
Let $p$ be a binary word with $3$ runs that is of length $l \geq 4$.
\begin{itemize}
    \item If $p$ is not trivially equivalent to $\one\one\zero\one$, we have $B_{n, p}(M_{n, p} - 1) = 0$ for all $n \geq l + 2$. 
    In particular, $p$ has an internal zero at $n$ for all $n \geq l + 2$.
    
    \item If $p$ is trivially equivalent to $\one\one\zero\one$, then $p$ does not have an internal zero at $6 = l + 2$, but for $n \geq 7$, we have $B_{n, p}(M_{n, p} - 1) = 0$.
\end{itemize}
\end{proposition}

\begin{proof}
Using trivial equivalences, we can assume $p = \one^i \zero^j \one^k$ for some $i, j, k \geq 1$. 
We will first show that it is enough to prove that no word of length $n$ with $3$ runs has $M_{n, p} - 1$ occurrences of $p$.

We first assume that $j \geq 2$. 
In this case, we show that a binary word of length $n$ that has more than $3$ runs can have at most $M_{n, p} - 2$ occurrences of $p$. 
Let $w$ be a binary word of length $n$ with $b$ $\zero$s that has more than $3$ runs. 
If the first or last letter of $w$ is not $\one$ then it cannot have more than $M_{n, p} - 2$ occurrences of $p$. 
This is because, using \Cref{maxform}, we can show that $M_{n - 1, p} \leq M_{n, p} - 2$. 
In fact, we can assume that the first run of $w$ is of size at least $i$ and the last is of size at least $k$.

Since $w$ has more than $3$ runs, we can select a $\one$ in $w$ that has $\zero$s both before and after it. 
Hence, we can choose $j$ $\zero$s in $w$ in such a way that some $\zero$s before this $\one$ as well as some $\zero$s after this $\one$ are chosen. 
This means that $c_p(w) \leq (\binom{b}{j} - 1)M_{n - b, \one^i\zero^k} + M_{n - b - 1, \one^i\zero^k}$. 
Note that we have $c_p(w) \leq \binom{b}{j}M_{n - b, \one^i\zero^k} \leq M_{n, p}$. 
First assume that $w$ has at least $i + k + 2$ letters equal to $\one$. 
Using the fact that $M_{m, \one^i\zero^k} - M_{m - 1, \one^i\zero^k} \geq 2$ for $m \geq i + k + 2$, we can show that $c_p(w) \leq M_{n, p} - 2$. 
If $w$ has $i + k + 1$ letters equal to $\one$, then it is of the form
\begin{equation*}
    \one^i \quad \zero^a \quad \one \quad \zero^{b - a} \quad \one^k
\end{equation*}
for some $a \in [1, b - 1]$. 
Without loss of generality, assume $i \geq k$. 
It can be checked that the word $\one^{i + 1}\zero^{b}\one^k$ has at least $2$ more occurrences of $p$ than $w$. 
Hence, $c_p(w) \leq M_{n, p} - 2$.

Next, we assume that $j = 1$. 
Let $w$ be a word of length $n$ with $b$ $\zero$s that has more than $3$ runs. 
Just as before we can assume that the first run of $w$ consists of at least $i$ $\one$s and the last consists of at least $k$ $\one$s. 
The number of occurrences of $p$ is a sum of $b$ terms of the form $\binom{l}{i}\binom{r}{k}$ where $l + r \leq n - b$. 
This is because each $\zero$ that has $l$ $\one$s before it and $r$ $\one$s after it contributes $\binom{l}{i}\binom{r}{k}$ occurrences of $p$.

First suppose $n - b \geq i + k + 2$. 
In this case, if the contribution of any of the $\zero$s is not $M_{n - b, \one^i\zero^k}$, then we have seen in \Cref{2runsmaxminus1} that it is at most $M_{n - b, \one^i\zero^k} - 2$. 
Hence, the maximum number of occurrences of $p$ in a word with $b$ $\zero$s is attained only when all terms are equal to $M_{n - b, \one^i\zero^k}$ and otherwise it is at least $2$ less than this. 
Hence, it is enough to show that $M_{n, p} - 1$ occurrences of $p$ is not possible in any word with $3$ runs.

Next, suppose $n - b = i + k + 1$. 
Just as before, this means that $w$ is of the form
\begin{equation*}
    \one^i \quad \zero^a \quad \one \quad \zero^{b - a} \quad \one^k
\end{equation*}
for some $a \in [1, b - 1]$. 
Without loss of generality, we assume $i \geq k$. 
It can be checked that if $i \geq 3$, then the word $w' = \one^{i + 2} \zero^{b - 1} \one^k$ has at least two more occurrences of $p$ than $w$. 
If $i = 2$, then $p = \one\one\zero\one$ or $p = \one\one\zero\one\one$. 
Even in these cases, if $n \geq l + 3$, then $c_p(w') \geq c_p(w) + 2$. 
For $n = l + 2$, one can check that $\one\one\zero\one$ does not have internal zeroes whereas $p = \one\one\zero\one\one$ satisfies $B_{n, p}(M_{n, p} - 1) = 0$.

We now show that no binary word with $3$ runs can have exactly $M_{n, p} - 1$ occurrences of $p$. 
Using \Cref{lcprop} and the same ideas as \Cref{2runsmaxminus1}, we have to show that if $\one^a\zero^b\one^c$ is $p$-optimal of length at least $l + 2$, then
\begin{equation*}
    \binom{a}{i}\left[\binom{b}{j}\binom{c}{k} - \binom{b - 1}{j}\binom{c + 1}{k}\right]
\end{equation*}
and similar expressions cannot be equal to $1$. 
We have already seen that this is not possible if either $j \geq 2$ or $k \geq 2$. 
If $j = k = 1$, then since $\one^a\zero^b\one^c$ is $p$-optimal and of length at least $l + 2$, one can show that we must have $a \geq i + 1$. 
This gives us the required result.
\end{proof}

We also note the following result, which is a consequence of the first part of the above proof.

\begin{result}\label{3runallopt3run}
Let $p$ be a binary word with $3$ runs whose second run has size at least $2$. 
Then any $p$-optimal word also has $3$ runs.
\end{result}

\subsection{Log-concavity and unimodality}

If $p$ has an internal zero at $n$, then the sequence $(B_{n, p}(k))_{k \geq 0}$ can not be unimodal. 
Hence, by \Cref{IZgenres}, the only binary words with at most $3$ runs for which these sequences might be unimodal are the alternating words. 
Clearly, $(B_{n, \texttt{1}}(k))_{k \geq 0}$ is the sequence $(\binom{n}{k})_{k \geq 0}$ which is not only unimodal, but also log-concave. 
For $p = \texttt{10}, \texttt{101}$ and $n \geq 4$, by \Cref{0count,1count,2count}, we have $B_{n, p}(0) > B_{n, p}(1) < B_{n, p}(2)$ and hence the sequence $(B_{n, p}(k))_{k \geq 0}$ is not unimodal.

\begin{proposition}\label{lcprop}
Let $A = (a_i)_{i \geq 0}$ and $B = (b_i)_{i \geq 0}$ be log-concave sequences of non-negative numbers without internal zeroes. 
For each $n \geq 0$ and $i \in [0, n]$ define $c_n(i) = a_ib_{n - i}$, $C_n = (c_n(i))_{i = 0}^n$, and $M_n = \operatorname{max}(C_n)$. 
The following hold for any $n \geq 0$.
\begin{itemize}
    \item The sequence $C_n$ is log-concave and does not have internal zeroes.
    \item If $c_n(i) = c_n(i + 1)$ for some $i \in [0, n - 1]$, then $M_n = c_n(i)$.
    \item If $M_n = c_n(i)$ for some $i \in [0, n]$, then $M_{n + 1} = \operatorname{max}\{c_{n + 1}(i), c_{n + 1}(i + 1)\}$.
    \item The sequence $(M_n)_{n \geq 0}$ is log-concave and does not have internal zeroes.
\end{itemize}
\end{proposition}

\begin{proof}
Log-concavity of $C_n$ is a direct consequence of the log-concavity of the sequences $A$ and $B$. 
The remaining claims can be proved using the following which hold for any $n \geq 0$ and $i \in [0, n - 1]$.
\begin{itemize}
    \item $c_n(i) < c_n(i + 1)$ implies $c_{n + 1}(i) < c_{n + 1}(i + 1)$
    \item $c_n(i) > c_n(i + 1)$ implies $c_{n + 1}(i + 1) > c_{n + 1}(i + 2)$
    \item $c_n(i) = c_n(i + 1)$ implies $c_{n + 1}(i) \leq c_{n + 1}(i + 1)$ and $c_{n + 1}(i + 1) \geq c_{n + 1}(i + 2)$
\end{itemize}
We omit the proofs of the above statements as well as how they imply the proposition since they mostly consist of case-by-case checks.
\end{proof}

Together with \Cref{maxform}, this gives us the following result. 
Note that any word with at most $3$ runs is trivially equivalent to one of the form mentioned in the following result.

\begin{theorem}
Let \textup{$p = \texttt{1}^i \texttt{0}^j \texttt{1}^k$} for some $i, k \geq 0$ and $j \geq 1$. 
We have the following:
\begin{itemize}
    \item The sequence $(M_{n, p})_{n \geq 0}$ is log-concave.
    \item If $\one^a\zero^b\one^c$ is $p$-optimal, so is at least one word in $\{\one^{a + 1}\zero^b\one^c, \one^a\zero^{b + 1}\one^c, \one^a\zero^b\one^{c + 1}\}$.
\end{itemize}
\end{theorem}

\section{Concluding remarks}\label{concsec}

In this section, we discuss some directions for future research.

\begin{enumerate}[leftmargin = 5mm]
\setlength{\itemsep}{5mm}

\item Just as in permutation patterns, two binary words $p$ and $q$ are said to be \emph{Wilf-equivalent} if $B_{n, p}(0) = B_{n, q}(0)$ for all $n \geq 0$. 
From \Cref{0count}, we see that any two words of the same length are Wilf-equivalent. 
A more interesting notion is that of \emph{strong Wilf-equivalence}.

\begin{definition}
Two binary words $p$ and $q$ are said to be \emph{strongly Wilf-equivalent} if $B_{n, p}(k) = B_{n, q}(k)$ for all $n, k \geq 0$.
\end{definition}

If $p$ and $q$ are strongly Wilf-equivalent, it is clear that their lengths must be the same. 
From \Cref{l+1}, we see that they must also have the same number of runs of each size. 
We have already noted that if $p$ and $q$ are trivially equivalent, then they are strongly Wilf-equivalent. 
Computations suggest that these are the only strong Wilf-equivalences.

\begin{conjecture}
Two binary words $p$ and $q$ are strongly Wilf-equivalent if and only if they are trivially equivalent.
\end{conjecture}

We have verified the above conjecture for binary words of length up to $13$ using Sage \cite{Sage}. 
There is a natural composition we can associate to a binary word using its runs. 
For example, to $\one\one\zero\zero\zero\one$, we associate the composition $(2, 3, 1)$. 
Words being trivially equivalent is equivalent to the associated compositions being equal or reverses of each other. 
From the expressions for $B_{n, p}(k)$ for $k \leq 4$ in \Cref{expsec}, we see that terms such as the number of $1$s in the composition, number of $2$s at the start or end (boundary) of the composition, etc. play a role. 
One would hope that the composition can be determined up to reversal if enough such terms play a role in expressions for higher values of $k$.

\item One obvious direction to extend the results in this paper would be to consider $m$-ary words for $m \geq 2$ instead of just binary words. 
Another possible direction is to refine the numbers $B_{n, p}(k)$ by keeping track of statistics such as number of $\one$s, number of runs, etc.

\item In this paper, we have studied binary words by focusing on their runs. 
Are there other viewpoints that can be used to obtain better results? 
For example, in \Cref{layeranalogue}, we have proved that if $p$ has at most $3$ runs, then for any length there exists a $p$-optimal word with the same number of runs as $p$. 
Is there a nicer (non-trivial) family of binary words $\mathcal{F}$ such that for any $p \in \mathcal{F}$, there exist $p$-optimal words of any given length in $\mathcal{F}$?

\item Another direction for future research could be finding expressions or generating functions for the numbers $B_{n, p}(k)$ for $k \geq 5$ as well as for $M_{n, p}$. 
This question might be easier to tackle for specific words $p$. 
For example, it is not too difficult to show that for any $i, j \geq 0$,
\begin{equation*}
    \sum_{n, k \geq 0} B_{n, \one^i \zero \one^j}(k) x^n t^k = \left(\frac{1}{1 - x}\right) \left(\sum_{k \geq 0} \prod_{l = 1}^k \frac{x}{1 - xt^{\binom{l}{i}\binom{k - l}{j}}}\right).
\end{equation*}
A particular class of words that might be interesting to study is the alternating binary words. 
These words are already interesting from the subsequence viewpoint since the alternating word of length $l$ has the highest number of distinct subsequences among all binary words of length $l$ (see \cite{maxdistsub}).

\end{enumerate}

\section*{Acknowledgements}
Krishna Menon is partially supported by a grant from the Infosys Foundation, and Anurag Singh is partially supported by the Research Initiation Grant from the Indian Institute of Technology Bhilai (No. 2009301). The computer algebra system SageMath \cite{Sage} provided valuable assistance in studying examples.

\bibliographystyle{abbrv}

\end{document}